\documentclass[a4paper,12pt]{article}

\usepackage{amsfonts}
\usepackage{amscd,color}
\usepackage{amsmath,amsfonts,amssymb,amscd}
\numberwithin{equation}{section}
\usepackage{indentfirst,graphicx,epsfig}
\usepackage{graphicx}
\input{epsf}
\usepackage{graphicx}
\usepackage{epstopdf}
\usepackage{caption}
\usepackage{subfigure}
\usepackage{mathrsfs}
\usepackage{ifpdf}
\usepackage{enumerate}
\usepackage{array}
\allowdisplaybreaks
\newtheorem{thm}{Theorem}[section]

\newtheorem{problem}[thm]{Problem}

\newtheorem{lem}[thm]{Lemma}

\newtheorem{claim}[thm]{Claim}

\newenvironment {proof} {\noindent{\textbf {Proof.}}}{\hfill$\Box$}
\newcommand{\ml}{l\kern-0.55mm\char39\kern-0.3mm}

\title{\textbf{The maximum spectral radius of $\theta_{1,3,3}$-free graphs with given size}}
\author{{\small Jing Gao, Xueliang Li} \\
{\small  Center for Combinatorics and LPMC}\\
{\small Nankai University, Tianjin 300071, China}\\
{\small gjing1270@163.com, lxl@nankai.edu.cn} \\
}

\date{}
\begin{document}
\maketitle
\begin{abstract}
A graph $G$ is said to be $F$-free if it does not contain $F$ as a subgraph. A theta graph, say $\theta_{l_1,l_2,l_3}$, is the graph obtained by connecting two distinct vertices with three internally disjoint paths of length $l_1, l_2, l_3$, where $l_1\leq l_2\leq l_3$ and $l_2\geq2$. Recently, Li, Zhao and Zou [arXiv:2409.15918v1] characterized the $\theta_{1,p,q}$-free graph of size $m$ having the largest spectral radius, where $q\geq p\geq3$ and $p+q\geq2k+1\geq7$, and proposed a problem on characterizing the graphs with the maximum spectral radius among $\theta_{1,3,3}$-free graphs. In this paper, we consider this problem and determine the maximum spectral radius of $\theta_{1,3,3}$-free graphs with size $m$ and characterize the extremal graph. Up to now, all the graphs in $\mathcal{G}(m,\theta_{1,p,q})$ which have the largest spectral radius have been determined, where $q\geq p\geq 2$.
\\[2mm]
\end{abstract}

\section{\bf Introduction}
For a simple graph $G = (V(G), E(G))$, we use $n := |G|=|V(G)|$ and $m := e(G)$ to denote the order and the size of $G$,
respectively. Since isolated vertices do not have an effect on the spectral radius of a graph, throughout this paper we consider graphs without isolated vertices. Let $N(v)$ or $N_G(v)$ be the set of neighbors of $v$, and $d(v)$ or $d_G(v)$ be the degree of a vertex $v$ in $G$. Denote $N[v]=N(v)\cup\{v\}$. For a subset $U\subseteq V(G)$, we denote by $N_U(u)$ the set of vertices of $U$ that are
adjacent to $u$, that is, $N_U(u) = N_G(u)\cap U$, and let $d_U(u)$ be the number of vertices of $N_U(u)$. For subsets $X$, $Y$ of $V(G)$, we write $E(X, Y)$ for the set of edges with one end in $X$ and the other in $Y$. Let $e(X, Y) =|E(X, Y)|$. If $Y=X$, we simply write $e(X)$ for $e(X, X)$. The distance between two distinct vertices $u, v\in V(G)$ is the length of a shortest path from $u$ to $v$ in $G$. The diameter $diam(G)$ of a graph $G$ is the greatest distance between any two vertices of $G$. The join of simple graphs $G$ and $H$, written $G\vee H$, is the graph obtained from the disjoint union $G\cup H$ by adding the edges to join every vertex of $G$ with every vertex of $H$. For graph notations and terminologies undefined here, readers are referred to \cite{Bondy-Murty-2008}.

Let $A(G)$ be the adjacency matrix of a connected graph $G$. The maximum of modulus of all eigenvalues of $A(G)$ is the spectral radius of $G$ and denoted by $\lambda(G)$.  Since $A(G)$ is irreducible and nonnegative for a connected graph $G$, by the Perron-Frobenius theorem there exists a unique positive unit eigenvector $\mathbf{x}$ corresponding to $\lambda(G)$, which is called Perron vector of $G$.

As usual, let $P_n$, $C_n$ and $K_{1,n-1}$ be the path, the cycle, and the star on $n$ vertices, respectively. Let $K_{1,n-1}+e$ be the graph obtained from $K_{1,n-1}$ by adding one edge within its independent set. A theta graph, say $\theta_{l_1,l_2,l_3}$, is the graph obtained by connecting two distinct vertices with three internally disjoint paths of length $l_1, l_2, l_3$, where $l_1\leq l_2\leq l_3$ and $l_2\geq2$.

Let $\mathcal{F}$ be a family of graphs. A graph $G$ is called $\mathcal{F}$-free if it does not contain any element in $\mathcal{F}$ as a subgraph. When the forbidden set $\mathcal{F}$ is a singleton, say ${F}$, then we write $F$-free for $\mathcal{F}$-free. Let $\mathcal{G}(m, \mathcal{F})$ denote the set of $\mathcal{F}$-free graphs with $m$ edges having no isolated vertices. If $\mathcal{F}=\{F\}$, then we write $\mathcal{G}(m, F)$ for $\mathcal{G}(m, \mathcal{F})$.

The classic Tur\'{a}n type problem asks what is the maximum number of edges in an $\mathcal{F}$-free graph of order $n$. In spectral graph theory, Nikiforov \cite{Nikiforov-2010} proposed a spectral Tur\'{a}n type problem which asks to determine the maximum spectral radius of an $\mathcal{F}$-free graph with $n$ vertices, which is known as the Brualdi-Solheid-Tur\'{a}n type problem. In the past few decades, this problem has been studied for many classes of graphs such as \cite{Nikiforov-2007, Nikiforov-2009, Nikiforov-2010, Zhai-Lin-2020, Zhai-Wang-2020}. In addition, Brualdi and Hoffman \cite{Brualdi-Hoffman-1985} raised another spectral Tur\'{a}n type problem: What is the maximal spectral radius of an $\mathcal{F}$-free graph with given size $m$? This problem is called the Brualdi-Hoffman-Tur\'{a}n type problem. Up to now, much attention has been paid to the Brualdi-Hoffman-Tur\'{a}n type problem for various families of graphs. For example, \cite{Nosal-1970} for $K_3$-free graphs, \cite{Lin-Ning-2021} for non-bipartite $K_3$-free graphs, \cite{Nikiforov-2002,Nikiforov-2006} for $K_{r+1}$-free graphs, \cite{Nikiforov-2009} for $C_4$-free graphs, \cite{Zhai-Lin-2021} for $K_{2,r+1}$-free graphs, \cite{Fang-You-2023} for non-star $K_{2,r+1}$-free graphs, \cite{Li-Zhai-2024} for $C_k^+$-free graphs where  $C_k^+$ is a graph on $k$ vertices obtained from $C_k$ by adding a chord between two vertices with distance two, \cite{Nikiforov} for $B_k$-free graphs where $B_k$ is obtained from $k$ triangles by sharing an edge, \cite{Zhang-Wang-2024} for $F_5$-free graphs, \cite{Li-Zhao-2024} for $F_{2k+2}$-free graphs where $F_k=K_1\vee P_{k-1}$, \cite{Li-Lu-2023} for $F_{2,3}$-free graphs, \cite{Li-Zhao-2024} for $F_{k,3}$-free graphs where $F_{k,3}$ is the friendship graph obtained from $k$ triangles by sharing a common vertex.

For theta graphs, Sun et al. \cite{Sun-Li-2023} established sharp upper bounds on spectral radius for $G$ in $\mathcal{G}(m,\theta_{1,2,3})$ and $\mathcal{G}(m,\theta_{1,2,4})$, respectively. Subsequently, Lu et al. \cite{Lu-Lu-2024} determined the graph among $\mathcal{G}(m,\theta_{1,2,5})$ having the largest spectral radius. Generally, Li et al. \cite{Li-Zhai-2024} confirmed the following theorem.
\begin{thm}\cite{Li-Zhai-2024}
Let $k\geq3$ and $m\geq4(k^2+3k+1)^2$. If $G\in\mathcal{G}(m,\theta_{1,2,2k-1})\cup \mathcal{G}(m,\theta_{1,2,2k})$, then
$$
\lambda(G)\leq \frac{k-1+\sqrt{4m-k^2+1}}{2},
$$
and equality holds if and only if $G\cong K_k\vee(\frac{m}{k}-\frac{k-1}{2})K_1$.
\end{thm}
Recently, Li et al. \cite{Li-Zhao-2024} determined the largest spectral radius of $\theta_{1,p,q}$-free graph with size $m$ for $q\geq p\geq3$ and $p+q\geq7$.
\begin{thm}\cite{Li-Zhao-2024}
Let $k\geq3$ and $m\geq\frac{9}{4}k^6+6k^5+46k^4+56k^3+196k^2$. If $G\in\mathcal{G}(m,\theta_{1,p,q})\cup \mathcal{G}(m,\theta_{1,r,s})$ with $p\geq q\geq3, s\geq r\geq3, p+q=2k+1$ and $r+s=2k+2$, then
$$
\lambda(G)\leq \frac{k-1+\sqrt{4m-k^2+1}}{2},
$$
and equality holds if and only if $G\cong K_k\vee(\frac{m}{k}-\frac{k-1}{2})K_1$.
\end{thm}
At the same time, they proposed the following problem in \cite{Li-Zhao-2024}.
\begin{problem}\cite{Li-Zhao-2024}
How can we characterize the graphs among $\mathcal{G}(m,\theta_{1,3,3})$ having the largest spectral radius?
\end{problem}

In this paper, we consider the above problem and characterize the unique graph with the maximum spectral radius among $\mathcal{G}(m, \theta_{1,3,3})$.

\begin{thm}\label{main theorem}
Let $G\in\mathcal{G}(m,\theta_{1,3,3})$ with $m\geq43$. Then $\lambda(G)\leq\frac{1+\sqrt{4m-3}}{2}$ and equality holds if and only if $G\cong K_2\vee\frac{m-1}{2}K_1$.
\end{thm}
\section{Preliminaries}
In this section, we introduce some basic lemmas which are useful in the subsequent sections.
\begin{lem}\cite{Nikiforov-2002,Nosal-1970}\label{lem:Nikiforov}
If $G\in \mathcal{G}(m, K_3)$, then $\lambda(G)\leq\sqrt{m}$. Equality holds if and only if $G$ is a complete bipartite graph.
\end{lem}

If $G$ is a bipartite graph with $m$ edges, then $G\in \mathcal{G}(m, K_{3})$. By Lemma \ref{lem:Nikiforov}, it follows that $\lambda(G) \leq\sqrt{m}$.

Wu et al. \cite{Wu-Xiao-2005} obtained a relationship between spectral radius of two graphs under graph operation, which plays an important role in our proofs.
\begin{lem}\cite{Wu-Xiao-2005}\label{lem:Wu-Xiao-2005}
Let $u$ and $v$ be two vertices of the connected graph $G$ with order $n$. Suppose $v_1,v_2,\ldots,v_s$ $(1\leq s\leq d_v)$ are some vertices of $N_G(v)\setminus N_G(u)$ and $\mathbf{x} = (x_1,x_2,\ldots,x_n)^T$ is the Perron vector of $G$, where $x_i$ corresponds to the vertex $v_i (1\leq i\leq n)$. Let $G'=G-\{vv_i|1\leq i\leq s\}+\{uv_i|1\leq i\leq s\}$. If $x_u\geq x_v$, then $\lambda(G)<\lambda(G')$.
\end{lem}

A cut vertex of a graph is a vertex whose deletion increases the number of components. A graph is called 2-connected, if it is a connected graph without cut vertex. Let $\mathbf{x}$ be the Perron vector of $G$ with coordinate $x_v$ corresponding to the vertex $v\in V(G)$. A vertex $u^\ast$ is said to be an extremal vertex if $x_{u^\ast}=\max\{x_u|u\in V(G)\}$.
\begin{lem}\cite{Zhai-Lin-2021}\label{lem:Zhai-Lin-2021}
Let $G$ be a graph in $\mathcal{G}(m,F)$ with the maximum spectral radius. If $F$ is a 2-connected graph and $u^\ast$ is an extremal vertex of $G$, then $G$ is connected and $d(u)\geq2$ for any $u\in V(G)\setminus N[u^\ast]$.
\end{lem}

\section{Proof of Theorem \ref{main theorem}}
Let $G^\ast$ be a graph in $\mathcal{G}(m,\theta_{1,3,3})$ with the maximum spectral radius. Note that $\lambda(K_2\vee\frac{m-1}{2}K_1)=\frac{1+\sqrt{4m-3}}{2}$ and $K_2\vee\frac{m-1}{2}K_1$ is $\theta_{1,3,3}$-free, we have
$$
\lambda(G^\ast)\geq\lambda(K_2\vee\frac{m-1}{2}K_1)=\frac{1+\sqrt{4m-3}}{2}.
$$
By Lemma \ref{lem:Zhai-Lin-2021}, we have $G^\ast$ is connected. Let $\lambda=\lambda(G^\ast)$ and $\mathbf{x}$ be the Perron vector of $G^\ast$ with coordinate $x_v$ corresponding to the vertex $v\in V(G^\ast)$. Assume that $u^\ast$ is the extremal vertex of $G^\ast$. That is, $x_{u^\ast}\geq x_u$ for any $u\in V(G)\setminus u^\ast$. Set $U=N_{G^\ast}(u^\ast)$ and $W=V(G^\ast)\setminus N_{G^\ast}[u^\ast]$. Let $W_H=N_W(V(H))$ for any component $H$ of $G^\ast[U]$. Since $G^\ast$ is $\theta_{1,3,3}$-free, $G^\ast[U]$ does not contain any path of length four and any cycle of length more than four.
\begin{lem}\label{pro of cycle four}
For any non-trivial component $H$ in $G^\ast[U]$, if $H$ contains a cycle of length four, then $N_W(u)\cap N_W(v)=\emptyset$ for any vertices $u$ and $v$ in the cycle of length four.
\end{lem}

\begin{proof}
Let the cycle $C_4$ in $H$ be $u_1u_2u_3u_4$. Suppose on the contrary that $N_W(u_i)\cap N_W(u_j)\neq\emptyset$ for some vertices $u_i$ and $u_j$, $1\leq i\neq j\leq4$. It follows that $N_W(u_i)\neq\emptyset$ and $N_W(u_j)\neq\emptyset$. Let $w\in N_W(u_i)\cap N_W(u_j)$. If $u_i$ and $u_j$ are adjacent in $C_4$. Without loss of generality, we assume that $u_i=u_1$ and $u_j=u_2$. Then $u^\ast u_1$, $u^\ast u_3u_4u_1$ and $u^\ast u_2wu_1$ are three internally disjoint paths of length 1,3,3 between $u^\ast$ and $u_1$. So $G^\ast$ contains $\theta_{1,3,3}$ as a subgraph, a contradiction. Hence, $u_i$ and $u_j$ are not adjacent in $C_4$. Assume that $u_i=u_1$ and $u_j=u_3$. It is easy to get that $u_1u_2$, $u_1u_4u^\ast u_2$ and $u_1wu_3u_2$ are three internally disjoint paths of length 1,3,3 between $u_1$ and $u_2$. Clearly, $G^\ast$ contains $\theta_{1,3,3}$ as a subgraph, a contradiction. This completes the proof.
\end{proof}

For convenience, we divide $U$ into two subsets $U_0$ and $U_+$ where $U_0$ is the set of isolated vertices of $G^\ast[U]$ and $U_+=U\setminus U_0$. It is easy to see that $m=|U|+e(U_+)+e(U,W)+e(W)$. Since $\lambda(G^\ast)\mathbf{x}=A(G^\ast)\mathbf{x}$, we have
$$
\lambda x_{u^\ast}=\sum_{u\in U}x_u=\sum_{u\in U_+}x_u+\sum_{u\in U_0}x_u.
$$
Furthermore, we can get
\begin{align*}
\lambda^2x_{u^\ast}&=\lambda(\lambda x_{u^\ast})=\lambda\sum_{u\in U}x_u\\
&=\sum_{u\in U}\sum_{v\in N_G(u)}x_v=\sum_{v\in V(G)}d_U(v)x_v\\
&=|U|x_{u^\ast}+\sum_{u\in U_+}d_U(u)x_u+\sum_{w\in W}d_U(w)x_w.
\end{align*}
Therefore,
\begin{align*}
(\lambda^2-\lambda)x_{u^\ast}&=|U|x_{u^\ast}+\sum_{u\in U_+}(d_U(u)-1)x_u+\sum_{w\in W}d_U(w)x_w-\sum_{u\in U_0}x_u\\
&\leq|U|x_{u^\ast}+\sum_{u\in U_+}(d_U(u)-1)x_u+e(U,W)x_{u^\ast}-\sum_{u\in U_0}x_u.
\end{align*}
Recall that $\lambda\geq\frac{1+\sqrt{4m-3}}{2}$, that is, $\lambda^2-\lambda\geq m-1=|U|+e(U_+)+e(U,W)+e(W)-1$. Hence
$$
\sum_{u\in U_+}(d_U(u)-1)x_u\geq\left(e(U_+)+e(W)+\sum_{u\in U_0}\frac{x_u}{x_{u^\ast}}-1\right)x_{u^\ast}.
$$
Let $\mathcal{H}$ be the set of all non-trivial components in $G^\ast[U]$. For each non-trivial component $H$ of $\mathcal{H}$, we denote $\eta(H):=\sum_{u\in V(H)}(d_H(u)-1)x_u$. Clearly,
\begin{align*}\label{inequality 1}
\sum_{H\in\mathcal{H}}\eta(H)\geq\left(e(U_+)+e(W)+\sum_{u\in U_0}\frac{x_u}{x_{u^\ast}}-1\right)x_{u^\ast},\tag{1}
\end{align*}
with equality if and only if $\lambda^2-\lambda= m-1$ and $x_w=x_{u^\ast}$ for any $w\in W$ with $d_U(w)\geq 1$.

\begin{claim}\label{no four cycle}
$G^\ast[U]$ contains no any cycle of length four.
\end{claim}

\begin{proof}
Suppose to the contrary that $G^\ast[U]$ contains $C_4$, that is, $G^\ast[U]$ contains a component $H\in\{C_4,C_4+e,K_4\}$ where $C_4+e$ is the graph obtained from $C_4$ by adding one edge to two nonadjacent vertices. Let $\mathcal{H'}$ be the family of components of $G^\ast[U]$ each of which contains $C_4$ as a subgraph, then $\mathcal{H}\setminus\mathcal{H'}$ be the family of other components of $G^\ast[U]$ each of which is a tree or unicyclic graph. Therefore, for each $H\in \mathcal{H}\setminus\mathcal{H'}$, we have
$$
\eta(H)=\sum_{u\in V(H)}(d_H(u)-1)x_u\leq(2e(H)-|H|)x_{u^\ast}\leq e(H)x_{u^\ast}.
$$
Next we show that
$$
\eta(H)<(e(H)-1)x_{u^\ast}+\frac{2\sum_{w\in W_H}x_w}{\lambda-3}
$$
for each $H\in \mathcal{H'}$. Let $H^\ast\in\mathcal{H'}$ with $V(H^\ast)=\{u_1,u_2,u_3,u_4\}$ and the cycle of length four be $u_1u_2u_3u_4$.

First, we consider the case $W_{H^\ast}=\emptyset$. Let $x_{u_1}=\max\{x_{u_i}|1\leq i\leq4\}$. Then
$$
\lambda x_{u_1}=\sum_{u\in N(u_1)}x_u\leq x_{u^\ast}+x_{u_2}+x_{u_3}+x_{u_4}\leq x_{u^\ast}+3x_{u_1}.
$$
Hence, $x_{u_1}\leq\frac{1}{\lambda-3}x_{u^\ast}$. Since $m\geq43$, we have $\lambda\geq\frac{1+\sqrt{4m-3}}{2}\geq7$. Thus, $x_{u_1}<\frac{1}{2}x_{u^\ast}$ and \begin{align*}
\eta(H^\ast)\leq(2e(H^\ast)-|H^\ast|)x_{u_1}< (e(H^\ast)-2)x_{u^\ast}<(e(H^\ast)-1)x_{u^\ast}+\frac{2\sum_{w\in W_{H^\ast}}x_w}{\lambda-3},
\end{align*}
as desired.

In the following, we assume that $W_{H^\ast}\neq\emptyset$. We consider the following two cases.

{\bf Case 1.} All vertices in $W_{H^\ast}$ have a unique common neighbor in $V(H^\ast)$.

Without loss of generality, let the common neighbor be $u_1$. It follows that $N_W(u_i)=\emptyset$ for $i\in\{2,3,4\}$. Let $x_{u_2}=\max\{x_{u_i}|2\leq i\leq4\}$. Then
$$
\lambda x_{u_2}\leq x_{u_1}+x_{u_3}+x_{u_4}+x_{u^\ast}\leq 2x_{u_2}+2x_{u^\ast}.
$$
Thus, $x_{u_2}\leq\frac{2}{\lambda-2}x_{u^\ast}\leq\frac{2}{5}x_{u^\ast}$ since $\lambda\geq7$. Therefore, we have
\begin{align*}
\eta(H^\ast)&=\sum_{u\in V(H^\ast)}(d_{H^\ast}(u)-1)x_u\\
&\leq(d_{H^\ast}(u_1)-1)x_{u_1}+(2e(H^\ast)-d_{H^\ast}(u_1)-3)x_{u_2}\\
&\leq\left(d_{H^\ast}(u_1)-1+\frac{4}{5}e(H^\ast)-\frac{2}{5}d_{H^\ast}(u_1)-\frac{6}{5}\right)x_{u^\ast}\\
&=\left(\frac{4}{5}e(H^\ast)+\frac{3}{5}d_{H^\ast}(u_1)-\frac{11}{5}\right)x_{u^\ast}.
\end{align*}
Since $d_{H^\ast}(u_1)\leq3$, the above inequality becomes
\begin{align*}
\eta(H^\ast)&\leq\left(\frac{4}{5}e(H^\ast)-\frac{2}{5}\right)x_{u^\ast}\\
&<(e(H^\ast)-1)x_{u^\ast}\\
&<(e(H^\ast)-1)x_{u^\ast}+\frac{2\sum_{w\in W_{H^\ast}}x_w}{\lambda-3}.
\end{align*}

{\bf Case 2.} There are at least two vertices of $W_{H^\ast}$ such that they have distinct neighbors in $V(H^\ast)$.

Since
\begin{align*}
\begin{cases}
\lambda x_{u_1}\leq x_{u_2}+x_{u_3}+x_{u_4}+x_{u^\ast}+\sum_{w\in N_{W_{H^\ast}}(u_1)}x_w,\\
\lambda x_{u_2}\leq x_{u_1}+x_{u_3}+x_{u_4}+x_{u^\ast}+\sum_{w\in N_{W_{H^\ast}}(u_2)}x_w,\\
\lambda x_{u_3}\leq x_{u_1}+x_{u_2}+x_{u_4}+x_{u^\ast}+\sum_{w\in N_{W_{H^\ast}}(u_3)}x_w,\\
\lambda x_{u_4}\leq x_{u_1}+x_{u_2}+x_{u_3}+x_{u^\ast}+\sum_{w\in N_{W_{H^\ast}}(u_4)}x_w,
\end{cases}
\end{align*}
we obtain
\begin{align*}
\lambda(x_{u_1}+x_{u_2}+x_{u_3}+x_{u_4})\leq3(x_{u_1}+x_{u_2}+x_{u_3}+x_{u_4})+4x_{u^\ast}+\sum_{i=1}^4\sum_{w\in N_{W_{H^\ast}}(u_i)}x_w.
\end{align*}
By Lemma \ref{pro of cycle four}, we get that $N_{W_{H^\ast}}(u_i)\cap N_{W_{H^\ast}}(u_j)=\emptyset$ for any vertices $u_i\neq u_j \in V(H^\ast)$. Thus, $\sum_{w\in W_{H^\ast}}x_w=\sum_{w\in N_W(V({H^\ast}))}x_w=\sum_{i=1}^4\sum_{w\in N_{W_{H^\ast}}(u_i)}x_w$. Therefore, by $\lambda\geq7$, we obtain
\begin{align*}
x_{u_1}+x_{u_2}+x_{u_3}+x_{u_4}&\leq\frac{4x_{u^\ast}}{\lambda-3}+\frac{\sum_{w\in W_{H^\ast}}x_w}{\lambda-3}\\
&\leq x_{u^\ast}+\frac{\sum_{w\in W_{H^\ast}}x_w}{\lambda-3}.
\end{align*}
Hence, by the definition of $\eta(H^\ast)$,
\begin{align*}
\eta(H^\ast)&\leq2(x_{u_1}+x_{u_2}+x_{u_3}+x_{u_4})\\
&\leq2x_{u^\ast}+\frac{2\sum_{w\in W_{H^\ast}}x_w}{\lambda-3}\\
&<(e(H^\ast)-1)x_{u^\ast}+\frac{2\sum_{w\in W_{H^\ast}}x_w}{\lambda-3}.
\end{align*}
Therefore, we conclude that $\eta(H)<(e(H)-1)x_{u^\ast}+\frac{2\sum_{w\in W_H}x_w}{\lambda-3}$ for each $H\in \mathcal{H'}$.
Recall that $\eta(H)\leq e(H)x_{u^\ast}$ for each $H\in \mathcal{H}\setminus\mathcal{H'}$. Thus,
\begin{align*}
\sum_{H\in\mathcal{H}}\eta(H)&=\sum_{H\in\mathcal{H'}}\eta(H)+\sum_{H\in\mathcal{H}\setminus\mathcal{H'}}\eta(H)\\
&<\sum_{H\in\mathcal{H'}}(e(H)-1)x_{u^\ast}+\sum_{H\in\mathcal{H'}}\frac{2\sum_{w\in W_H}x_w}{\lambda-3}+\sum_{H\in\mathcal{H}\setminus\mathcal{H'}}e(H)x_{u^\ast}\\
&=e(U_+)x_{u^\ast}-\sum_{H\in\mathcal{H'}}x_{u^\ast}+\sum_{H\in\mathcal{H'}}\frac{2\sum_{w\in W_H}x_w}{\lambda-3}.
\end{align*}
For any $H\in\mathcal{H'}$ satisfying $W_{H}=\emptyset$, we have $\sum_{w\in W_H}x_w=0$. For any $H\in\mathcal{H'}$ satisfying $W_{H}\neq\emptyset$ and any $w\in W_{H}$, since $G^\ast$ is $\theta_{1,3,3}$-free, we obtain $W_{H}\cap W_{G^\ast[U]\setminus H}=\emptyset$. Then $d_{U\setminus V(H)}(w)=0$. By Lemma \ref{pro of cycle four}, we have $d_H(w)=1$. It follows that $d_U(w)=1$. As $d(w)\geq2$ by Lemma \ref{lem:Zhai-Lin-2021}, it is easy to get that $d_W(w)\geq1$. Thus, $\sum_{H\in\mathcal{H'}}\sum_{w\in W_H}x_w\leq\sum_{H\in\mathcal{H'}}\sum_{w\in W_H}d_W(w)x_w\leq\sum_{H\in\mathcal{H'}}\sum_{w\in W_H}d_W(w)x_{u^\ast}\leq2e(W)x_{u^\ast}$. Note that $\lambda\geq7$. Therefore,
\begin{align*}
\sum_{H\in\mathcal{H}}\eta(H)&< e(U_+)x_{u^\ast}-\sum_{H\in\mathcal{H'}}x_{u^\ast}+\sum_{H\in\mathcal{H'}}\frac{2\sum_{w\in W_H}x_w}{\lambda-3}\\
&\leq e(U_+)x_{u^\ast}-\sum_{H\in\mathcal{H'}}x_{u^\ast}+\frac{4e(W)}{\lambda-3}x_{u^\ast}\\
&\leq (e(U_+)+e(W)-\sum_{H\in\mathcal{H'}}1)x_{u^\ast},
\end{align*}
which contradicts with (\ref{inequality 1}). Hence, $G^\ast[U]$ contains no $C_4$. This completes the proof.
\end{proof}

By Claim \ref{no four cycle}, we know that each non-trivial component of $G^\ast[U]$ is either a tree or a unicyclic graph $K_{1,r}+e$ with $r\geq2$. Let $c$ be the number of non-trivial tree-components of $G^\ast[U]$. Then
$$
\sum_{H\in\mathcal{H}}\eta(H)\leq\sum_{H\in\mathcal{H}}\sum_{u\in V(H)}(d_H(u)-1)x_{u^\ast}=\sum_{H\in\mathcal{H}}(2e(H)-|H|)x_{u^\ast}=(e(U_+)-c)x_{u^\ast}.
$$
Combining with (\ref{inequality 1}), we get
\begin{align*}\label{inequality 2}
e(W)\leq 1-c-\sum_{u\in U_0}\frac{x_u}{x_{u^\ast}}.\tag{2}
\end{align*}
Thus, $e(W)\leq1$ and $c\leq1$. In addition, if $e(W)=1$, then $c=0$, $U_0=\emptyset$, $\lambda^2-\lambda= m-1$, $x_w=x_{u^\ast}$ for any $w\in W$ with $d_U(w)\geq 1$ and $x_u=x_{u^\ast}$ for any $u\in V(H)$ with $d_H(u)\geq 2$.

\begin{claim}
$e(W)=0$.
\end{claim}

\begin{proof}
Suppose on the contrary that $e(W)=1$. Let $w_1w_2$ be the unique edge in $G^\ast[W]$. Note that $c=0$ and $U_0=\emptyset$, it follows that each component of $G^\ast[U]$ is isomorphic to a unicyclic graph $K_{1,r}+e$ with $r\geq2$. That is, each component of $G^\ast[U]$ contains a triangle. Let $H$ be a component of $G^\ast[U]$ and $u_1u_2u_3$ is the triangle $C_3$ of $H$. Since $G^\ast$ is $\theta_{1,3,3}$-free, we get that $d_{C_3}(w_1)+d_{C_3}(w_2)\leq3$. Otherwise, there are two cases. One case is $d_{C_3}(w_i)=3$ and $d_{C_3}(w_j)\geq1$ for $i\neq j\in\{1,2\}$. Without loss of generality, we suppose that $d_{C_3}(w_1)=3$ and $d_{C_3}(w_2)\geq1$. Let $u_1\in N_{C_3}(w_2)$. It is easy to find that $u_1u_2$, $u_1u^\ast u_3u_2$ and $u_1w_2w_1u_2$ are three internally disjoint paths of length 1,3,3 between $u_1$ and $u_2$. It is a contradiction. The other case is $d_{C_3}(w_i)\geq2$ for $i\in\{1,2\}$. Since $|C_3|=3$, we obtain $|N_{C_3}(w_1)\cap N_{C_3}(w_2)|\geq1$. Assume $u_1,u_p\in N_{C_3}(w_1)$ and $u_1,u_q\in N_{C_3}(w_2)$ with $p,q\in\{2,3\}$. If $p\neq q$, then $u_1u_q$, $u_1w_1w_2u_q$ and $u_1u^\ast u_pu_q$ are three internally disjoint paths of length 1,3,3 between $u_1$ and $u_q$. If $p= q$, for convenience, suppose $p=q=2$, then $u_1u_q$, $u_1w_1w_2u_q$ and $u_1u^\ast u_3u_q$ are three internally disjoint paths of length 1,3,3 between $u_1$ and $u_q$. It is also a contradiction. By Lemma \ref{lem:Zhai-Lin-2021} and $e(W)=1$, we get $d_U(w_i)\geq1$ for $i\in\{1,2\}$. Recall that $x_w=x_{u^\ast}$ for any $w\in W$ with $d_U(w)\geq 1$. We get $x_{w_1}=x_{w_2}=x_{u^\ast}$. As $d_H(u_i)\geq2$, we have $x_{u_i}=x_{u^\ast}$ for $i\in\{1,2,3\}$. This implies that
\begin{align*}
2\lambda x_{u\ast}&=\lambda x_{w_1}+\lambda x_{w_2}\\
&=x_{w_2}+\sum_{u\in N_U(w_1)}x_u+x_{w_1}+\sum_{u\in N_U(w_2)}x_u\\
&\leq x_{w_2}+x_{w_1}+\sum_{u\in N_{C_3}(w_1)}x_u+\sum_{u\in N_{C_3}(w_2)}x_u+2\sum_{u\in U\setminus C_3}x_u\\
&\leq x_{w_2}+x_{w_1}+3x_{u^\ast}+2(\lambda x_{u^\ast}-x_{u_1}-x_{u_2}-x_{u_3})\\
&=2x_{u^\ast}+3x_{u^\ast}+2(\lambda x_{u^\ast}-3x_{u^\ast})\\
&=2\lambda x_{u^\ast}-x_{u^\ast}.
\end{align*}
It is a contradiction for $x_{u^\ast}>0$. The proof is complete.
\end{proof}

\begin{claim}\label{no triangle}
$G^\ast[U]$ contains no triangle.
\end{claim}

\begin{proof}
Suppose on the contrary that $G^\ast[U]$ contains triangles. Then $G^\ast[U]$ contains a component which is isomorphic to $K_{1,r}+e$ with $r\geq2$. Let $H^\ast\cong K_{1,r}+e$ be a component of $G^\ast[U]$. It follows that $e(H^\ast)=r+1$. Suppose $u_1u_2u_3$ is the triangle of $H^\ast$ and $d_{H^\ast}(u_1)=d_{H^\ast}(u_2)=2$.

If $W_{H^\ast}=\emptyset$, then $x_{u_1}=x_{u_2}$. Hence,
\begin{align*}
\lambda x_{u_1}=x_{u_2}+x_{u_3}+x_{u^\ast}\leq x_{u_1}+2x_{u^\ast}.
\end{align*}
This implies that $x_{u_1}\leq\frac{2}{\lambda-1}x_{u^\ast}$. Therefore,
$$
\eta(H^\ast)=x_{u_1}+x_{u_2}+(r-1)x_{u_3}\leq\frac{4}{\lambda-1}x_{u^\ast}+(r-1)x_{u^\ast}.
$$
Since $m\geq43$ and $\lambda\geq7$, we get $\eta(H^\ast)<rx_{u^\ast}=(e(H^\ast)-1)x_{u^\ast}$. Note that $\eta(H)\leq e(H)x_{u^\ast}$ for any other component $H\in\mathcal{H}\setminus H^\ast$ of $G^\ast[U]$. Hence,
\begin{align*}
\sum_{H\in\mathcal{H}}\eta(H)&=\eta(H^\ast)+\sum_{H\in\mathcal{H}\setminus H^\ast}\eta(H)\\
&<(e(H)-1)x_{u^\ast}+\sum_{H\in\mathcal{H}\setminus H^\ast}e(H)x_{u^\ast}\\
&=(e(U_+)-1)x_{u^\ast},
\end{align*}
which contradicts with (\ref{inequality 1}). Thus, $W_{H^\ast}\neq\emptyset$.

Because $e(W)=0$, by Lemma \ref{lem:Zhai-Lin-2021}, we have $d_U(w)\geq2$ for any $w\in W_{H^\ast}$. Suppose $r\geq3$, let $u_4,\ldots,u_{r+1}$ be the neighbors of $u_3$. For $w\in W_{H^\ast}$, if $\{u_1,u_2\}\subseteq N_U(w)$, then $u_1u_3$, $u_1u^\ast u_4u_3$ and $u_1w u_2u_3$ are three internally disjoint paths of length 1,3,3 between $u_1$ and $u_3$, a contradiction. If $\{u_i,u_3\}\subseteq N_U(w)$, then $u_ju_3$, $u_ju^\ast u_4u_3$ and $u_ju_iwu_3$ are three internally disjoint paths of length 1,3,3 between $u_j$ and $u_3$ where $i\neq j\in\{1,2\}$, a contradiction. If $\{u_i,u_j\}\subseteq N_U(w)$, then $u_iu_3$, $u_iu^\ast u_{\{1,2\}\setminus i}u_3$ and $u_iwu_ju_3$ are three internally disjoint paths of length 1,3,3 between $u_i$ and $u_3$ where $i\in\{1,2\}$ and $j\in\{4,\ldots,r+1\}$, a contradiction. If $\{u_3,u_j\}\subseteq N_U(w)$, then $u^\ast u_3$, $u^\ast u_jwu_3$ and $u^\ast u_1u_2u_3$ are three internally disjoint paths of length 1,3,3 between $u^\ast$ and $u_3$ where $j\in\{4,\ldots,r+1\}$, a contradiction. If $\{u_i,u_j\}\subseteq N_U(w)$, then $u^\ast u_j$, $u^\ast u_iwu_j$ and $u^\ast u_1u_3u_j$ are three internally disjoint paths of length 1,3,3 between $u^\ast$ and $u_j$ where $i\neq j\in\{4,\ldots,r+1\}$, a contradiction. If $\{u_i,v\}\subseteq N_U(w)$, then $u^\ast u_i$, $u^\ast vwu_i$ and $u^\ast u_4u_3u_i$ are three internally disjoint paths of length 1,3,3 between $u^\ast$ and $u_i$ where $i\in\{1,2\}$ and $v\in U\setminus V(H^\ast)$, a contradiction. If $\{u_3,v\}\subseteq N_U(w)$, then $u^\ast u_3$, $u^\ast vwu_3$ and $u^\ast u_1u_2u_3$ are three internally disjoint paths of length 1,3,3 between $u^\ast$ and $u_3$ where $v\in U\setminus V(H^\ast)$, a contradiction. If $\{u_i,v\}\subseteq N_U(w)$, then $u^\ast u_i$, $u^\ast vwu_i$ and $u^\ast u_1u_3u_i$ are three internally disjoint paths of length 1,3,3 between $u^\ast$ and $u_i$ where $i\in\{4,\ldots,r+1\}$ and $v\in U\setminus V(H^\ast)$, a contradiction. Therefore, $r=2$. That is, $H^\ast$ is a triangle $u_1u_2u_3$.

First, we assume that $|W_{H^\ast}|=1$. Let $W_{H^\ast}=\{w\}$. Since $G^\ast$ is $\theta_{1,3,3}$-free, it follows that $d_{U\setminus V(H^\ast)}(w)=0$. Therefore, $d(w)=d_{H^\ast}(w)$. As $H^\ast$ is a triangle, we obtain $d(w)=2$ or $d(w)=3$. If $d(w)=2$, without loss of generality, we suppose $N(w)=\{u_1,u_2\}$. Then $x_{u_1}=x_{u_2}$. Since
$$
\lambda x_{u_3}=x_{u_1}+x_{u_2}+x_{u^\ast}\leq 3x_{u^\ast},
$$
we obtain $x_{u_3}\leq\frac{3}{\lambda}x_{u^\ast}$. Furthermore, $$\
\lambda x_{u_1}=x_{u_2}+x_{u_3}+x_{u^\ast}+x_w\leq x_{u_1}+\frac{3}{\lambda}x_{u^\ast}+2x_{u^\ast}.
$$
This implies that $x_{u_1}\leq\frac{3+2\lambda}{\lambda(\lambda-1)}x_{u^\ast}$. Thus,
$$
\eta(H^\ast)=x_{u_1}+x_{u_2}+x_{u_3}\leq\frac{7\lambda+3}{\lambda(\lambda-1)}x_{u^\ast}.
$$
Since $\frac{7x+3}{x(x-1)}$ is decreasing in variable $x$ and $\lambda\geq7$, we get
$$
\eta(H^\ast)\leq\frac{7\times7+3}{7\times6}x_{u^\ast}<2x_{u^\ast}=(e(H^\ast)-1)x_{u^\ast}.
$$
Recall that $\eta(H)\leq e(H)x_{u^\ast}$ for any other component $H\in\mathcal{H}\setminus H^\ast$ of $G^\ast[U]$. Hence,
\begin{align*}
\sum_{H\in\mathcal{H}}\eta(H)&=\eta(H^\ast)+\sum_{H\in\mathcal{H}\setminus H^\ast}\eta(H)\\
&<(e(H^\ast)-1)x_{u^\ast}+\sum_{H\in\mathcal{H}\setminus H^\ast}e(H)x_{u^\ast}\\
&=(e(U_+)-1)x_{u^\ast},
\end{align*}
which contradicts with (\ref{inequality 1}). If $d(w)=3$, that is, $N(w)=\{u_1,u_2,u_3\}$, then $x_{u_1}=x_{u_2}=x_{u_3}$. By
$$
\lambda x_{u_1}=x_{u_2}+x_{u_3}+x_{u^\ast}+x_w\leq 2x_{u_1}+2x_{u^\ast},
$$
we obtain $x_{u_1}\leq\frac{2}{\lambda-2}x_{u^\ast}$. Therefore, by $\lambda\geq7$,
$$
\eta(H^\ast)=x_{u_1}+x_{u_2}+x_{u_3}\leq\frac{6}{\lambda-2}x_{u^\ast}<2x_{u^\ast}=(e(H^\ast)-1)x_{u^\ast}.
$$
Thus, $\sum_{H\in\mathcal{H}}\eta(H)<(e(U_+)-1)x_{u^\ast}$, a contradiction. So $|W_{H^\ast}|\geq2$.

Similarly, since $G^\ast$ is $\theta_{1,3,3}$-free, we have $d_{U\setminus V(H^\ast)}(w)=0$ for any $w\in W_{H^\ast}$. Therefore, $2\leq d(w)=d_{H^\ast}(w)\leq3$. If there is a vertex $w'\in W_{H^\ast}$ such that $d(w')=3$, then $N(w')=\{u_1,u_2,u_3\}$. Note that $|W_{H^\ast}|\geq2$, there exists a vertex $w''\neq w'$ in $W_{H^\ast}$ satisfying $d(w'')\geq2$. Suppose that $u_1,u_2\in N(w'')$. Then $u^\ast u_1$, $u^\ast u_3w'u_1$ and $u^\ast u_2w''u_1$ are three internally disjoint paths of length 1,3,3 between $u^\ast$ and $u_1$, a contradiction. Hence, $d(w)=d_{H^\ast}(w)=2$ for any $w\in W_{H^\ast}$. This implies that $1\leq|N(w)\cap N(w')|\leq2$ for any two vertices $w,w'\in W_{H^\ast}$. If $|N(w)\cap N(w')|=1$, without loss of generality, we assume that $N(w)=\{u_1,u_2\}$ and $N(w')=\{u_1,u_3\}$. It is easy to see that $u^\ast u_1$, $u^\ast u_3w'u_1$ and $u^\ast u_2wu_1$ are three internally disjoint paths of length 1,3,3 between $u^\ast$ and $u_1$, a contradiction. Therefore, $|N(w)\cap N(w')|=2$. That is, $N(w)=N(w')$ for any $w,w'\in W_{H^\ast}$. Without loss of generality, we suppose that $N(w)=\{u_1,u_2\}$ for any $w\in W_{H^\ast}$. Let $G'$ be a graph such that $V(G')=V(G^\ast)$ and $E(G')=E(G^\ast)-\{u_1w|w\in N_{W}(u_1)\}+\{u^\ast w|w\in N_{W}(u_1)\}$. One can verify that $G'$ is $\theta_{1,3,3}$-free. By Lemma \ref{lem:Wu-Xiao-2005}, we have $\lambda(G')>\lambda$. It is a contradiction with the maximality of $G^\ast$. We complete the proof.
\end{proof}

{\bf Proof of Theorem \ref{main theorem}.} By Claims \ref{no four cycle} and \ref{no triangle}, we have that each component of $G^\ast[U]$ is a non-trivial tree or an isolated vertex. By inequality (\ref{inequality 2}), the number $c$ of non-trivial tree-components is at most 1. If $c=0$, then $G^\ast$ is bipartite. By Lemma \ref{lem:Nikiforov}, $\lambda\leq\sqrt{m}<\frac{1+\sqrt{4m-3}}{2}$, a contradiction. Hence $c=1$. It follows that $U_0=\emptyset$. Let $H$ be the unique component of $G^\ast[U]$. That is, $G^\ast[U]\cong H$ and $W=W_H$. Since $G^\ast$ is $\theta_{1,3,3}$-free, $diam(H)\leq3$.

If $diam(H)=3$, then $H$ is a double star. Denote the two centers of $H$ by $u_1$ and $u_2$. If $W_{H}=\emptyset$, then $G^\ast=u^\ast\vee H$.  Without loss of generality, suppose that $x_{u_1}\geq x_{u_2}$. Let $G'$ be a graph such that $V(G')=V(G^\ast)$ and $E(G')=E(G^\ast)-\{u_2v|v\in N_{H}(u_2)\setminus u_1\}+\{u_1v|v\in N_{H}(u_2)\setminus u_1\}$. One can verify that $G'$ is $\theta_{1,3,3}$-free. By Lemma \ref{lem:Wu-Xiao-2005}, we have $\lambda(G')>\lambda$, a contradiction. If $W_{H}\neq\emptyset$, then $N(w)=\{u_1,u_2\}$ for any $w\in W_{H}$. Otherwise, $G^\ast$ contains $\theta_{1,3,3}$ as a subgraph, a contradiction. Let $G''$ be a graph such that $V(G'')=V(G^\ast)$ and $E(G'')=E(G^\ast)-\{u_2w|w\in W_{H}\}+\{u^\ast w|w\in W_{H}\}$. Obviously, $G''$ is $\theta_{1,3,3}$-free. By Lemma \ref{lem:Wu-Xiao-2005}, we have $\lambda(G'')>\lambda$, a contradiction. Hence, $diam(H)\leq2$. That is, $H\cong K_{1,r}$ with $r\geq1$.

Let $V(H)=\{u_0,u_1,\ldots,u_r\}$ and $u_0$ be the center of $H$ with $r\geq1$. Since
$$
\lambda x_{u_0}=x_{u_1}+x_{u_2}+\cdots+x_{u_r}+x_{u^\ast}+\sum_{w\in N_W(u_0)}x_w,
$$
and
$$
\lambda x_{u^\ast}=x_{u_0}+x_{u_1}+x_{u_2}+\cdots+x_{u_r},
$$
we obtain $\lambda(x_{u_0}-x_{u^\ast})=x_{u^\ast}+\sum_{w\in N_W(u_0)}x_w-x_{u_0}$. Note that $x_{u_0}\leq x_{u^\ast}$ and $x_v>0$ for $v\in V(G^\ast)$. Thus, $N_W(u_0)=\emptyset$ and $x_{u_0}=x_{u^\ast}$. If $W\neq\emptyset$, by Lemmma \ref{lem:Zhai-Lin-2021}, we have $d(w)\geq2$ for $w\in W_{H}$. Note that $e(W)=0$. Let $w_0\in W$. Suppose $u_1,u_2\in N_{H}(w_0)$. If $r\geq3$, then $u^\ast u_2$, $u^\ast u_1w_0u_2$ and $u^\ast u_ru_0u_2$ are three internally disjoint paths of length 1,3,3 between $u^\ast$ and $u_2$, a contradiction. So $r\leq2$. Since $d(w)\geq2$ and $N_W(u_0)=\emptyset$, we obtain $r\neq1$. Therefore, $r=2$ and $N(w)=\{u_1,u_2\}$ for any $w\in W$. Hence, $x_{u_0}=x_{u^\ast}$ and $x_{u_1}=x_{u_2}$. As
$$
\lambda x_{u^\ast}=x_{u_0}+x_{u_1}+x_{u_2}= x_{u^\ast}+2x_{u_1},
$$
it follows that $x_{u_1}=\frac{\lambda-1}{2}x_{u^\ast}$. Note that $x_{u^\ast}\geq x_{u_1}$. We can get $\lambda\leq3$, it is a contradiction with $\lambda\geq7$. Thus $W=\emptyset$. Equivalently, $G^\ast\cong K_1\vee K_{1,r}$ with $2r+1=m$. Hence $G^\ast\cong K_2\vee \frac{m-1}{2}K_1$. This completes the proof.\hfill$\Box$
\section*{\bf Acknowledgments}

This work was supported by National Natural Science Foundation of China (Nos.12131013 and 12161141006).\\


\begin{thebibliography}{1}
\bibitem{Bondy-Murty-2008}
J. Bondy, U. Murty, Graph Theory, Springer, New York, 2008.

\bibitem{Brualdi-Hoffman-1985}
R. Brualdi, A. Hoffman, On the spectral radius of (0,1)-matrices, Linear Algebra Appl. 65 (1985) 133--146.

\bibitem{Fang-You-2023}
X. Fang, L. You, The maximum spectral radius of graphs of given size with forbidden subgraph, Linear Algebra Appl. 666 (2023) 114--128.


\bibitem{Li-Lu-2023}
Y. Li, L. Lu, Y. Peng, Spectral extremal graphs for the bowtie, Discrete Math. 346 (2023) 113680.

\bibitem{Li-Zhai-2024}
X. Li, M. Zhai, J. Shu, A Brualdi-Hoffman-Tur\'{a}n problem on cycles, European J. Combin. 120 (2024) 103966.

\bibitem{Li-Zhao-2024}
S. Li, S. Zhao, L. Zou, Spectral extrema of graphs with fixed size: forbidden fan graph, friendship graph or theta graph, arXiv:2409.15918v1.

\bibitem{Lin-Ning-2021}
H. Lin, B. Ning, B. Wu, Eigenvalues and triangles in graphs, Combin. Probab. Comput. 30(2)(2021) 258--270.

\bibitem{Lu-Lu-2024}
J. Lu, L. Lu, Y. Li, Spectral radius of graphs forbidden $C_7$ or $C_6^\triangle$, Discrete Math. 347 (2024) 113781.

\bibitem{Nikiforov-2002}
V. Nikiforov, Some inequalities for the largest eigenvalue of a graph, Combin. Probab. Comput. 11 (2002) 179--189.

\bibitem{Nikiforov-2006}
V. Nikiforov, Walks and the spectral radius of graphs, Linear Algebra Appl. 418 (2006) 257--268.

\bibitem{Nikiforov-2007}
V. Nikiforov, Bounds on graph eigenvalues II, Linear Algebra Appl. 427 (2007) 183--189.

\bibitem{Nikiforov-2009}
V. Nikiforov, The maximum spectral radius of $C_4$-free graphs of given order and size, Linear Algebra Appl. 430 (2009) 2898--2905.

\bibitem{Nikiforov-2010}
V. Nikiforov, The spectral radius of graphs without paths and cycles of specified length, Linear Algebra Appl. 432 (2010) 2243--2256.

\bibitem{Nikiforov}
V. Nikiforov, On a theorem of Nosal, arXiv: 2104.12171.

\bibitem{Nosal-1970}
E. Nosal, Eigenvalues of graphs, Master's thesis, University of Calgary, 1970.

\bibitem{Sun-Li-2023}
W. Sun, S. Li, W. Wei, Extensions on spectral extrema of $C_5/C_6$-free graphs with given size, Discrete Math. 346 (2023) 113591.

\bibitem{Wu-Xiao-2005}
B. Wu, E. Xiao, Y. Hong, The spectral radius of trees on $k$ pendant, Linear Algebra Appl. 395 (2005) 343--349.

\bibitem{Zhai-Lin-2020}
M. Zhai, H. Lin, Spectral extrema of graphs: forbidden hexagon, Discrete Math. 343 (2020) 112028.


\bibitem{Zhai-Lin-2021}
M. Zhai, H. Lin, J. Shu, Spectral extrema of graphs with fixed size: cycles and complete bipartite graphs, European J. Combin. 95 (2021) 103322.

\bibitem{Zhai-Wang-2020}
M. Zhai, B. Wang, L. Fang, The spectral Tur\'{a}n problem about graphs with no 6-cycle, Linear Algebra Appl. 590 (2020) 22--31.

\bibitem{Zhang-Wang-2024}
Y. Zhang, L. Wang, On the spectral radius of graphs without a gem, Discrete Math. 347 (2024) 114171.


\end{thebibliography}
\end{document}